\documentclass[12pt]{amsart}
\usepackage[latin1]{inputenc}
\usepackage{amsxtra,amssymb,amsthm,amsmath,amscd,mathrsfs, epsfig, eufrak}
\usepackage[all]{xy}
\def \F {{\mathbb F}}

\def \Q {{\mathbb Q}}

\def \Z {{\mathbb Z}}

\def\tr{\mathop{\rm tr}\nolimits}
\def\frakp{{\mathfrak{p}}}

\def \d {\,{\rm d}}

\def\GL{\hbox{{\rm GL}}}
\def\Li{\hbox{{\rm Li}}}

\def\leq{\leqslant}
\def\geq{\geqslant}
\def\le{\leqslant}
\def\ge{\geqslant}

\theoremstyle{plain}
\newtheorem{theorem}{Theorem}[section]
\newtheorem{lemma}[theorem]{Lemma}
\newtheorem{corollary}[theorem]{Corollary}

\newtheorem{hypothesis}[theorem]{Hypothesis}

\newtheorem{proposition}[theorem]{Proposition}

\theoremstyle{remark}
\newtheorem{remark}{Remark}

\numberwithin{equation}{section}

\begin{document}

\title[Pseudoprime reductions of Elliptic curves]
{Pseudoprime reductions of Elliptic curves}
\author{C. David \& J. Wu}
\address{Department of Mathematics and Statistics,
Concordia University,
1455 de Mai\-sonneuve West,
Montr\'eal, QC, H3G 1M8, Canada}
\email{cdavid@mathstat.concordia.ca}

\address{Institut Elie Cartan Nancy
\\
CNRS, Universit\'e Henri Poincar\'e (Nancy 1), INRIA
\\
Boulevard des Aiguillettes, B.P. 239
\\
54506 Van\-d\oe uvre-l\`es-Nancy
\\
France}
\email{wujie@iecn.u-nancy.fr}

\date{\today}

\begin{abstract}
Let $E$ be an elliptic curve over $\F_p$ without complex multiplication, and for each prime
$p$ of good reduction, let $n_E(p) = \left| E(\F_p) \right|$. Let $Q_{E,b}(x)$ be the number of primes $p \leq x$
such that $b^{n_E(p)} \equiv b\,({\rm mod}\,n_E(p))$, and $\pi_{E, b}^{\rm pseu}(x)$ be the number of
{\it compositive} $n_E(p)$ such that $b^{n_E(p)} \equiv b\,({\rm mod}\,n_E(p))$ (also called
elliptic curve pseudoprimes). Motivated by cryptography applications, we address in this paper
the problem of finding upper bounds for $Q_{E,b}(x)$ and $\pi_{E, b}^{\rm pseu}(x)$,
generalising some of the literature for the classical pseudoprimes \cite{Erdos56, Pomerance81} to this
new setting.
\end{abstract}
\maketitle


\section{Introduction}

\smallskip

The study of the structure and size of the group of points
of elliptic curves over finite fields has received much attention since Koblitz and Miller independently  proposed in 1985
elliptic curve cryptography,
an approach to public-key cryptography based on the algebraic structure of elliptic curves over finite fields.
Those cryptosystems guarantee, in general,
a high level of security with less cost in the size of the keys, whenever
the order of the group has a big prime divisor.

Let $E$ be an elliptic curve defined over $\Q$ with conductor $N_E$ and without complex multiplication (CM),
and denote by $E(\F_p)$ the reduction of $E$ modulo $p$.
Writing $n_E(p):=|E(\F_p)|$,
it is an interesting problem to study the asymptotic behavior of
\begin{equation}\label{defpiEtwinx}
\pi_E^{\rm twin}(x)
:= \big|\big\{p\le x \, : \, n_E(p)\;\hbox{is prime}\big\}\big|.
\end{equation}
Here and in the sequel, the letters $p$, $q$ and $\ell$ denote prime numbers.
Koblitz \cite{Ko88} conjectured that as $x\to\infty$,
\begin{equation}\label{KoblitzC}
\pi_E^{\rm twin}(x)
\sim \frac{C_E^{\rm twin}x}{(\log x)^2},
\end{equation}
with an explicit constant $C_E^{\rm twin}$ depending only on $E$
(see \cite[(2.5)]{DW10} for its precise definition).
It is easy to see that if $C_E^{\rm twin}=0$, then $\pi_E^{\rm twin}(x)\ll_E 1$ for all $x\ge 1$.
The asymptotic formula \eqref{KoblitzC} can be regarded as the analogue of the twin prime conjecture for elliptic curves.
As in the classical case, Koblitz's conjecture is still open,
but was shown to be true on average over all elliptic curves
\cite{BCD08}. One can also apply sieve methods to get unconditional or conditional upper bounds
for $\pi_E^{\rm twin}(x)$.
The best unconditional upper bound is due to Zywina \cite[Theorem 1.3]{zywina-largesieve}, and the
best bound under the Generalised Riemann Hypothesis (GRH) is due to David \& Wu \cite[Theorem 2]{DW10}.
For $E$ an elliptic curve over $\Q$ without CM, and
for any $\varepsilon>0$, those bounds are
\begin{equation}\label{UB.piEtwinx}
\pi_E^{\rm twin}(x)
\le \begin{cases}
\displaystyle
(24C_E^{\rm twin}+\varepsilon) \frac{x}{(\log x)\log_2x} & \text{(unconditionally)},
\\\noalign{\vskip 2mm}
\displaystyle
(10C_E^{\rm twin}+\varepsilon) \frac{x}{(\log x)^2}          & \text{(under the GRH)},
\end{cases}
\end{equation}
where $\log_k$ denotes the $k$-fold logarithm function.

Let $b\ge 2$ be an integer.
We say that a composite positive integer $n$ is a pseudoprime to base $b$
if the congruence
\begin{equation}\label{qnq}
b^n\equiv b\,({\rm mod}\,n)
\end{equation}
holds.
In practice, primality testing algorithms are not fast when one wants 
to test many numbers in a short amount of time,
and
pseudoprime testing can provide a quick pre-selection procedure to get rid of most of the pretenders.
The distribution of pseudoprimes was studied by many authors, including \cite{Erdos56, Pomerance81}.
Motivated by applications in cryptography,
the question of the distribution of pseudoprimes in certain sequences of positive integers has received some interest
(see \cite{CLS09, GoPo91, LuSh06, LuSh07, PoRo80}).
In particular
Cojocaru, Luca \& Shparlinski \cite{CLS09} have investigated
distribution of pseudoprimes in $\{n_E(p)\}_{p\,\text{primes}}$.
Define
$$
Q_{E, b}(x)
:= \big|\big\{p\le x \, : \, b^{n_E(p)}\equiv b\,({\rm mod}\,n_E(p))\big\}\big|.
$$
According to Fermat's little theorem, if $n_E(p)$ is a prime such that $n_E(p)\nmid b$,
then \eqref{qnq} holds with $n=n_E(p)$. Thus
\begin{equation}\label{}
\pi_E^{\rm twin}(x)
\le Q_{E, b}(x)
\end{equation}
for all $x\ge 2$.
Cojocaru, Luca \& Shparlinski \cite[Theorems 1 and 2]{CLS09} proved that
for any fixed base $b\ge 2$ and elliptic curve $E$ without CM,
the estimates
\begin{equation}\label{CLS}
Q_{E, b}(x)\ll_{E, b}
\begin{cases}
\displaystyle
\frac{x(\log_3x)^2}{(\log x)\log_2x}   &  \text{(unconditionally)}
\\\noalign{\vskip 2mm}
\displaystyle
\frac{x(\log_2x)^2}{(\log x)^2}            &  \text{(under the GRH)}
\end{cases}
\end{equation}
hold for all $x\ge 10$, where the implied constant depends on $E$ and $b$.
\footnote{We noticed that there are two inaccuracies in Cojocaru, Luca \& Shparlinski's proof of \eqref{CLS}:
With the notation of \cite{CLS09},
we have $t_b(\ell)\mid (n_E(p)-1)$ instead of $t_b(\ell)\mid n_E(p)$ (see \cite[page 519]{CLS09}).
Thus the inequality (see \cite[page 520]{CLS09})
$$
\#{\mathcal T}\le \sum_{y<\ell\le z} \Pi(x; \ell\rho(t_b(\ell)))
$$
does not hold.
Secondly the statements of Lemmas 3, 4, 6 and 7 of \cite{CLS09} are not true when $(m, M_E) \neq 1$
(see Section \ref{CDT} for the definition of $M_E$).
Then, the proofs of Lemma 9 and 10 hold only for $(m, M_E) = 1$.
This is not sufficient for the proof bounding $\#{\mathcal T}$ since $t_b(\ell)$ is not necessarily coprime with $M_E$.}

\smallskip

The first aim of this paper is to improve \eqref{CLS}.

\begin{theorem}\label{UB.Pseudoprime.thm}
Let $E$ be an elliptic curve over $\Q$ without CM and $b\ge 2$ be an integer.
For any $\varepsilon>0$, we have
\begin{equation}\label{UB.QEbx}
Q_{E, b}(x)
\le \begin{cases}
\displaystyle
(48{\rm e}^\gamma + \varepsilon) \frac{x\log_3x}{(\log x)\log_2x} & \text{$($unconditionally$)$}
\\\noalign{\vskip 2mm}
\displaystyle
(28{\rm e}^\gamma + \varepsilon) \frac{x\log_2x}{(\log x)^2}          & \text{$($under the GRH$)$}
\end{cases}
\end{equation}
for all $x\ge x_0(E, b, \varepsilon)$,
where $\gamma$ is the Euler constant.
\end{theorem}

Denoting by $\pi(x)$ the number of primes not exceeding $x$, and by
$\pi_b^{\rm pseu}(x)$ the number of pseudoprimes to base $b$
not exceeding $x$,
then it is known that (see \cite{Erdos56, Pomerance81})
\begin{equation}\label{ErdosPomerance}
\pi_b^{\rm pseu}(x)=o(\pi(x))
\end{equation}
as $x\to\infty$.
Precisely Pomerance \cite[Theorem 2]{Pomerance81} proved that
\footnote{In \cite{Pomerance81}, the definition of pseudoprime to base $b$ is slightly stronger:
$b^{n-1}\equiv 1 \, ({\rm mod}\,n)$ in place of $b^n\equiv b\,({\rm mod}\,n)$.
It is easy to adapt Pomerance's proof of \cite[Theorem 2]{Pomerance81} to obtain \eqref{Pomerance1},
as we do in this paper for the context of elliptic curves pseudoprimes. See Section \ref{Section-adapt-Pomerance} for more details.}
\begin{equation}\label{Pomerance1}
\pi_b^{\rm pseu}(x)\le \frac{x}{\sqrt{L(x)}}
\end{equation}
for $x\ge x_0(b)$, where
\begin{equation}\label{defLx}
L(x) := {\rm e}^{(\log x)(\log_3x)/\log_2x}.
\end{equation}
As analogue of $\pi_b^{\rm pseu}(x)$ for elliptic curve, we introduce
$$
\pi_{E, b}^{\rm pseu}(x)
:= \big|\big\{p\le x \, : \, \text{$n_E(p)$ is pseudoprime to base $b$}\big\}\big|.
$$
Clearly
$$
Q_{E, b}(x)
= \pi_E^{\rm twin}(x) + \pi_{E, b}^{\rm pseu}(x).
$$
In view of \eqref{ErdosPomerance}, it seems reasonable to conjecture
\begin{equation}\label{PseuTwin}
\pi_{E, b}^{\rm pseu}(x)
= o\big(\pi_E^{\rm twin}(x)\big)
\end{equation}
as $x\to\infty$.

In order to establish analogue of \eqref{Pomerance1} for $\pi_{E, b}^{\rm pseu}(x)$,
we need a supplementary hypothesis.

\begin{hypothesis}\label{H}
Let $E$ be an elliptic curve over $\Q$.
There is a positive constant $\delta$ such that
\begin{equation}\label{H.delta}
M_E(n):=\sum_{p\le x, \, n_E(p)=n} 1\ll_E n^{\delta}
\end{equation}
holds uniformly for $n\ge 1$ and $x\ge 1$,
where the implied constant can depend on the elliptic curve $E$.
\end{hypothesis}

By the Hasse bound $|p+1-n_E(p)|\le 2\sqrt{p}$, it is easy to see that
\begin{equation}\label{Hasse1}
n_E(p)/16\le p\le 16n_E(p)
\end{equation}
for all $p$.
Thus the relation $n_E(p)=n$ and the Hasse bound imply that $|p-n|\le 9\sqrt{n}$.
Therefore \eqref{H.delta} holds trivially with $\delta=\tfrac{1}{2}$ and an absolute implicit constant.
It is conjectured that \eqref{H.delta} should hold for any $\delta>0$ (see \cite[Question 4.11]{Kowalski06}).
Kowalski proved that this conjecture is true for elliptic curves with CM \cite[Proposition 5.3]{Kowalski06}
and on average for elliptic curves without CM \cite[Lemma 4.10]{Kowalski06}.

The next theorem shows that we can obtain a better conditional upper bound for $\pi_{E, b}^{\rm pseu}(x)$
than $\pi_E^{\rm twin}(x)$,
which can be regarded as analogue of \eqref{Pomerance1} for elliptic curves without CM.

\begin{theorem}\label{UB.Pseudoprime.NonCM}
Let $E$ be an elliptic curve over $\Q$ without CM and $b\ge 2$ be an integer.
If we assume the GRH and Hypothesis \ref{H} with $\delta<\tfrac{1}{24}$,
we have
\begin{equation}\label{UB.piEbpseux.NonCM}
\pi_{E, b}^{\rm pseu}(x)
\le \frac{x}{L(x)^{1/40}}
\end{equation}
for all $x\ge x_0(E, b, \delta)$.
\end{theorem}

In view of Koblitz's conjecture \eqref{KoblitzC},
the result of Theorem \ref{UB.Pseudoprime.NonCM} then encourages our belief in Conjecture \eqref{PseuTwin}.

By combining \eqref{UB.piEbpseux.NonCM} and  the second part of \eqref{UB.piEtwinx},
we immediately get the following result.

\begin{corollary}\label{QEbx.NonCM}
Let $E$ be an elliptic curve over $\Q$ without CM and $b\ge 2$ be an integer.
If we assume the GRH and hypothesis \ref{H} with $\delta<\tfrac{1}{24}$,
for any $\varepsilon>0$ we have
\begin{equation}\label{UB.QEbx.NonCM}
Q_{E, b}(x)
\le (10C_E^{\rm twin}+\varepsilon) \frac{x}{(\log x)^2}
\end{equation}
for all $x\ge x_0(E, b, \delta, \varepsilon)$.
\end{corollary}

We can also consider the same problem for elliptic curves with CM.
In this case, we easily obtain an unconditional result
by using the bound \eqref{Pomerance1} of Pomerance for pseudoprimes and
a result of Kowalski \cite{Kowalski06} about the second moment of $M_E(n)$ for elliptic curves with CM.

\begin{theorem}\label{UB.Pseudoprime.CM}
Let $E$ be an elliptic curve over $\Q$ with CM and $b\ge 2$ be an integer.
Then we have
\begin{equation}\label{UB.piEbpseux.CM}
\pi_{E, b}^{\rm pseu}(x)
\le \frac{x}{L(x)^{1/4}}
\end{equation}
for all $x\ge x_0(E, b)$.
\end{theorem}

It seems be interesting to prove that
\begin{equation}\label{LB.piEpseu.LTC.CM}
\pi_{E, b}^{\rm pseu}(x)\to\infty,
\quad\text{as $x\to\infty$.}
\end{equation}
We hope to come back to this question in the future.

\vskip 2mm

{\bf Acknowledgments}.
This first author was supported by the Natural Sciences and Engineering Research Council
of Canada (Discovery Grant 155635-2008) and by a grant to the Institute for Advanced Study
from the Minerva Research Foundation during the academic year 2009-2010.
The second author wishes to thank the Centre de Recherches Math\'ematiques (CRM) in Montr\'eal
for hospitality and support during the preparation of this article.

\vskip 10mm

\section{Chebotarev density theorem}
\label{CDT}

In order to prove Theorems \ref{UB.Pseudoprime.thm} and \ref{UB.Pseudoprime.NonCM},
we need to know some information on the distribution of  the sequence $\{n_E(p)\}_{p\,\text{primes}}$
in arithmetic progressions.
The aim of this section is to give such results with the help of the Chebotarev density theorem.
Our main result of this section is Theorem \ref{thm-upperbound} below.

We conserve all notation of \cite[Sections 2 and 3]{DW10}.
In particular, for an elliptic curve $E$ without complex multiplication defined over the rationals,
let $E[n]$ be the group of $n$-torsion points of $E$, and let
let $L_n$ be the field extension obtained from $\Q$
by adding the coordinates of the $n$-torsion points of $E$. This is a Galois extension of $\Q$,
and we denote $G(n):=\mbox{Gal}(L_n/\Q)$. Since $E[n](\bar{\Q}) \simeq \Z/n \Z \times \Z/n \Z$,
choosing a basis for the $n$-torsion and looking at the
action of the Galois automorphisms on the $n$-torsion, we get an
injective homomorphism
$$
\rho_n : G(n) \hookrightarrow \mbox{GL}_2(\Z/n \Z).
$$
If $p \nmid nN_E$, then $p$ is unramified in $L_n/\Q.$ Let $p$ be an
unramified prime, and let $\sigma_p$ be the Artin symbol of $L_n/\Q$
at the prime $p$.  For such a prime $p$, $\rho_{n}(\sigma_p)$ is a
conjugacy class of matrices of $\mbox{GL}_2(\Z/n \Z)$. Since the
Frobenius endomorphism $(x,y) \mapsto (x^p, y^p)$ of $E$ over $\F_p$
satisfies the polynomial $x^2 - a_E(p) x + p$,
it is not difficult to see that
$$
\tr(\rho_{n}(\sigma_p)) \equiv a_E(p)\,({\rm mod}\,n) \qquad{\rm
and}\qquad \det(\rho_{n}(\sigma_p)) \equiv p\,({\rm mod}\,n).
$$
To study the sequence $\left\{ n_E(p) \right\}_{p\, {\rm primes}}$, we will use
the Chebotarev Density Theorem to count the number of primes $p$ such that
$$
n_E(p)
= p + 1 - a_E(p)
\equiv \det(\rho_{n}(\sigma_p)) + 1 - \tr(\rho_{n}(\sigma_p))
\equiv r \,({\rm mod}\,n)
$$
for integers $r, n$ with $n \geq 2$. We then define
$$C_r(n) = \left\{ g \in G(n) \, : \,
\det(g)+1-\tr(g)\equiv r\,({\rm mod}\,n) \right\}.$$
Then, the $C_r(n)$ are unions of conjugacy classes in $G(n)$.
We also denote $C(n) := C_0(n)$.
For any prime $\ell$ such that $(\ell, M_E) = 1$, $G(\ell) = \mbox{GL}_2(\Z/\ell\Z)$,
and it is easy to compute that
\begin{equation}\label{SizeCrell}
| C_r(\ell) | =  \begin{cases}
\displaystyle \ell(\ell^2 - 2)  & \mbox{for $r \equiv 0\,({\rm mod}\,\ell)$}
\\\noalign{\vskip 1,5mm}
 \displaystyle \ell(\ell^2 - \ell - 1)   & \mbox{for $r \equiv 1\,({\rm mod}\,\ell)$}
\\\noalign{\vskip 1,5mm}
\displaystyle \ell(\ell^2 - \ell - 2)  & \mbox{for $r \not\equiv 0,1\,({\rm mod}\,\ell)$}
\end{cases}
\end{equation}
and then
\begin{equation}\label{SizeCrelloverGell}
\frac{| C_r(\ell) |}{|G(\ell)|} =  \begin{cases}
\displaystyle \frac{\ell^2 - 2}{(\ell-1)^2 (\ell+1)}  & \mbox{for $r \equiv 0\,({\rm mod}\,\ell)$}
\\\noalign{\vskip 1,5mm}
 \displaystyle\frac{\ell^2 - \ell - 1}{(\ell-1)^2 (\ell+1)}  & \mbox{for $r \equiv 1\,({\rm mod}\,\ell)$}
\\\noalign{\vskip 1,5mm}
\displaystyle \frac{\ell^2 - \ell - 2}{(\ell-1)^2 (\ell+1)}  & \mbox{for $r \not\equiv 0,1\,({\rm mod}\,\ell)$}.
\end{cases}
\end{equation}

It was shown by Serre \cite{Se72} that the Galois groups $G(n) \subseteq \mbox{GL}_2(\Z/n\Z)$
are large, and that there exists a
positive integer $M_E$ depending only on the elliptic curve $E$ such that
\begin{eqnarray}
\label{serre1} &&\mbox{If $(n, M_E)=1$, then $G(n) =
\mbox{GL}_2(\Z/n \Z)$;}
\\
\label{serre2} &&\mbox{If $(n, M_E)=(n,m)=1$, then $G(mn) \simeq
G(m) \times G(n)$;}
\\
\label{serre3} && \mbox{If $M_E \mid m$, then $G(m) \subseteq \mbox{GL}_2(\Z/m \Z)$
is the
full inverse image of} \\ \nonumber
&& \mbox{$G(M_E) \subseteq \mbox{GL}_2(\Z/M_E \Z)$ under the projection  map.}
\end{eqnarray}

Let
\begin{eqnarray*}
\pi_{C_r(n)}(x, L_n/\Q)
:= \left|\left\{ p \leq x : p \nmid nN_E \;\;\mbox{and}\;\; \rho_n(\sigma_p) \in  C_r(n) \right\}\right|.
\end{eqnarray*}
The following proposition (with a better error term) was proved in \cite[Theorem 3.9]{DW10} for the conjugacy class $C(n) = C_0(n) \subseteq G(n)$
when $n$ is squarefree, and can be easily generalised to general $n$ and
$r$.

\begin{proposition}\label{CDT-bothcases}
Let $E$ be an elliptic curve over $\Q$ without CM. Let $r\ge 0$ be an integer,
and let $n=dm$ be any positive integer with $(d, M_E)=1$ and $m\mid {M_E}^\infty$.
\footnote{The notation $d\mid n^\infty$ means that $p\mid d\,\Rightarrow\, p\mid n$
and the notation $p^k\| n$ means that $p^k\mid n$ and $p^{k+1}\nmid n$.}
\par
{\rm (i)}
Then,
$$
\pi_{C_r(n)}(x, L_{n}/\Q)
=  \frac{|C_r(m)|}{|G(m)|}  \bigg( \prod_{\ell^k\| d} \frac{|C_r(\ell^k)|}{|\GL_2(\Z / \ell^k \Z)|}  \bigg) \Li(x)
+ O_E\Big( x \exp\Big\{\!- A n^{-2} \! \sqrt{\log{x}}\Big\}\!\Big)
$$
uniformly for $\log{x} \gg n^{12} \log{n}$,
where the implied constants depend only on the elliptic curve $E$
and $A$ is a positive absolute constant.
\par
{\rm (ii)}
Assuming the GRH for the Dedekind zeta functions of the number fields $L_n/\Q$,
we have
$$
\pi_{C_r(n)}(x, L_{n}/\Q)
=  \frac{|C_r(m)|}{|G(m)|} \bigg( \prod_{\ell^k\| d} \frac{|C_r(\ell^k)|}{|\GL_2(\Z / \ell^k \Z)|}  \bigg) \Li(x)
+ O_E\left(n^{3}x^{1/2} \log{(n  x)}\right).
$$
\end{proposition}

\begin{proof}
To prove (i) and (ii), one applies the effective Cheboratev Density Theorem due to Lagarias and
Odlyzko \cite{LaOd} and slightly improved by Serre in \cite{Se81}, as stated in \cite[Theorem 3.1]{DW10}
with the appropriate bounds for the discriminants of number fields \cite[Proposition 6]{Se81},
and the bound of Stark \cite{St74} for the exceptional zero of Dedekind $L$-functions for (i).
We refer the reader to \cite{DW10} for more details.
\end{proof}

\begin{remark}
There are many cases where we can improve the error term in Proposition \ref{CDT-bothcases} (ii)
by applying a strategy first used in \cite{Se81} and \cite{MMS88} to reduce to the case of an
extension where Artin's conjecture holds. The error term then becomes
$$
 O_E
\big(n^{3/2}x^{1/2} \log{(n x)}\big).
$$
This can be done if $r=0$ (as in \cite[Theorem 3.9]{DW10}), or if $(n, M_E)=1$ for any $r$.
To apply the strategy of \cite{Se81} and \cite{MMS88} and obtain this improved error
term, one needs to
insure that $C_r(n) \cap B(n) \neq \emptyset$, where $B(n)$ is the Borel subgroup of $\mbox{GL}_2(\Z/n \Z)$.
For example,
this is the case if $E$ is a Serre curve,
and most elliptic curves are Serre curves as it was shown by Jones \cite{JonesSC}.
\end{remark}

We now need upper and lower bounds on the size of the main term of Proposition \ref{CDT-bothcases},
which are computed in the next lemma.

\begin{lemma} \label{lemma-upper-lower-bounds}
Let $E$ be an elliptic curve over $\Q$ without CM.
For all primes $\ell\nmid M_E$ and integers $k\ge 1$, we have the bounds
\begin{equation}\label{usefulinequality2}
\frac{1}{\varphi(\ell^k)} \cdot \frac{\ell-2}{\ell-1}
\le \frac{|C_r(\ell^k)|}{|\GL_2(\Z / \ell^k \Z)|}
\le \frac{1}{\varphi(\ell^k)}
\end{equation}
when $r \not\equiv 0  \,({\rm mod}\,\ell)$, and the bounds
\begin{equation}\label{usefulinequality3}
\frac{1}{\varphi(\ell^k)} \cdot \frac{\ell-2}{\ell-1}
\le \frac{|C_r(\ell^k)|}{|\GL_2(\Z / \ell^k \Z)|}
\le \frac{1}{\varphi(\ell^k)} \bigg(1+\frac{1}{(\ell^3-1)(\ell^2-1)}\bigg)
\end{equation}
when $r \equiv 0  \,({\rm mod}\,\ell)$.
\par

Furthermore, for $m \mid {M_E}^\infty$ such that $|C_r(m)|\not=0$,
we have that
\begin{equation} \label{bound2}
\frac{1}{\varphi(m)} \ll_E \frac{|C_r(m)|}{|G(m)|} \ll_E \frac{1}{\varphi(m)}
\end{equation}
with constants depending only on the elliptic curve $E$.
In particular, the upper bound in \eqref{bound2} holds without the hypothesis $|C_r(m)|\not=0$.
\end{lemma}

\begin{proof}
Fix $\ell \nmid M_E$ and $k\ge 1$.
To count the number of elements in
$C_r(\ell^k)$, we count
the matrices $\tilde{g}  \in \GL_2(\Z / \ell^k \Z)$ which are the inverse images of a matrix $g \in C_r(\ell)$ under the
projection map from $\GL_2(\Z / \ell^k \Z)$ to $\GL_2(\Z / \ell \Z)$, and
which satisfy
$$
\det(\tilde{g})+1-\tr{(\tilde{g})} \equiv r  \,({\rm mod}\,n).
$$
Let
$$
g=
\begin{pmatrix}
a & b
\\
c & d
\end{pmatrix},
\qquad
\tilde{g} =
\begin{pmatrix}
\tilde{a} & \tilde{b}
\\
\tilde{c} & \tilde{d}
\end{pmatrix}.
$$
If $b \not\equiv 0   \,({\rm mod}\,\ell)$, then $\tilde{b}$ is invertible,
and  we have to count the number of $\tilde{a}, \tilde{b}, \tilde{c}, \tilde{d}$
lifting $a,b,c,d$ such that
$$
\tilde{c} \equiv \frac{\tilde{a} \tilde{d} - (\tilde{a} + \tilde{d}) - r + 1}{\tilde{b}}  \,({\rm mod}\,\ell^k),
$$
and there are $\ell^{3(k-1)}$ such lifts. Similarly if $c \not\equiv 0   \,({\rm mod}\,\ell)$.

If $a \not\equiv 1  \,({\rm mod}\,\ell)$, then $\tilde{a}-1$ is invertible, and
we have to count the number of $\tilde{a}, \tilde{b}, \tilde{c}, \tilde{d}$
lifting $a,b,c,d$ such that
$$
\tilde{d} \equiv \frac{\tilde{a} r + \tilde{b} \tilde{c} -1 + \tilde{a}}{\tilde{a}-1}  \,({\rm mod}\,\ell^k),
$$
and there are $\ell^{3(k-1)}$ such lifts. Similarly if $d \not\equiv 1   \,({\rm mod}\,\ell)$.
This proves \eqref{usefulinequality2} as  the identity matrix does
not belong to $C_r(\ell)$ when $r \not\equiv 0  \,({\rm mod}\,\ell)$.
Then, the number of lifts of any matrix from $C_r(\ell)$ to $C_r(\ell^k)$
is $\ell^{3(k-1)}$, and the number of lifts
from $\GL_2(\Z / \ell \Z)$ to $\GL_2(\Z / \ell^k \Z)$ is $\ell^{4(k-1)}$, which gives
$$
\frac{|C_r(\ell^k)|}{|\GL_2(\Z / \ell^k \Z)|} =  \frac{\ell^{3(k-1)} |C_r(\ell)|}{\ell^{4(k-1)}
|\GL_2(\Z / \ell \Z)|},
$$
and the result follows by using \eqref{SizeCrelloverGell}.

Finally, we have to count the number of lifts
$$
\begin{pmatrix}
1 + k_1 \ell & k_2 \ell
\\
k_3 \ell & 1 + k_4 \ell
\end{pmatrix}
$$
of the identity matrix such that $\ell^2 (k_1 k_4 - k_2 k_3) \equiv r \,({\rm mod}\,\ell^k)$,
where $0 \leq k_i < \ell^{k-1}$.
We assume that $k \geq 2$.
If $r \not\equiv 0 \,({\rm mod}\,\ell^2)$, there are no lifts, and there are $\ell^4$ lifts
if $r \equiv 0 \,({\rm mod}\,\ell^2)$.
Let $v = \min_{i} v_\ell(k_i)$,
where $v_\ell(n)$ is the $\ell$-adic evaluation of $n$,
and write
$k_i  = \ell^v k_i'$ with $0 \leq k_i' < \ell^{k-1-v}$. If $r \not\equiv 0 \,({\rm mod}\,\ell^{2+v})$,
there is no solution with $k_1, k_2, k_3, k_4$ such that $v = \min_{i} v_\ell(k_i)$.
Suppose that $r \equiv 0 \,({\rm mod}\,\ell^{2+v})$. Then we need to solve
$$
\ell^{2+v} (k_1' k_4' - k_2' k_3') \equiv \ell^{2+v} r' \,({\rm mod}\,\ell^{k}) \iff
(k_1' k_4' - k_2' k_3') \equiv   r' \,({\rm mod}\,\ell^{k-2-v}).
$$
Without loss of generality, $v_\ell(k_1')=0$, and
$$
k_4'  \equiv   \frac{r'+k_2'k_3'}{k_1'}  \,({\rm mod}\,\ell^{k-2-v}),
$$
and there are $\ell \ell^{3(k-1-v)}$ solutions $k_1', k_2', k_3', k_4'$.
The number of lifts of the identity matrix is then bounded by
\begin{equation} \label{boundforI}
\ell \sum_{v=0}^{k-2} \ell^{3(k-1-v)}
= \ell \ell^{3(k-1)} \sum_{v=0}^{k-2} \ell^{-3v}
\leq \ell \ell^{3(k-1)} \frac{\ell^3}{\ell^3-1}\cdot
\end{equation}

We now prove \eqref{usefulinequality3}.
Using \eqref{boundforI} and the first formula of \eqref{SizeCrell}, it follows that
\begin{align*}
\frac{\ell^{k-1} |C_r(\ell^k)|}{|\GL_2(\Z / \ell^k \Z)|}
& \leq \frac{|C_r(\ell)|}{\GL_2(\Z / \ell \Z)}
+ \frac{\ell^4/(\ell^3 - 1)}{\GL_2(\Z / \ell \Z)}
\\
& = \frac{(\ell^3-1)(\ell^2-1) + 1}{(\ell-1)(\ell^2-1)(\ell^3-1)}\cdot
\end{align*}
For the lower bound, we have
$$
\frac{\ell^{k-1} |C_r(\ell^k)|}{|\GL_2(\Z / \ell^k \Z)|}
\geq \frac{|C_r(\ell)|-1}{\GL_2(\Z / \ell \Z)}
= \frac{\ell (\ell^2-2)-1}{\ell (\ell-1)(\ell^2-1)}
\geq \frac{\ell-2}{(\ell-1)^2}\cdot
$$

We now prove \eqref{bound2}.
Let $m' = \prod_{p \mid m} p^{\min{(v_p(m), v_p(M_E))}}$,
where $v_p(m)$ is the $p$-adic evaluation of $m$. By (\ref{serre3}), $G(m)$ is the
full inverse image of $G(m')$ under the projection map from $\mbox{GL}_2(\Z/m\Z)$ to $\mbox{GL}_2(\Z/m'\Z)$.
Fix $g \in C_r(m')$, and we now count the number of lifts $\tilde{g}$ in $C_r(m)$. By the
Chinese Remainder Theorem, it suffices to count the number of lifts from  $C_r(p^{v_p(m')})$ to $C_r(p^{v_p(m)})$
for each $p \mid m$. In general, fix $1 \leq e \leq k$, fix $g \in \mbox{GL}_2(\Z/p^e\Z)$
such that $\det(g) + 1 - \tr(g) \equiv r \,({\rm mod}\,p^{e})$, and we count the number of lifts
$\tilde{g} \in \mbox{GL}_2(\Z/p^k\Z)$ such that  $\det(\tilde{g}) + 1 - \tr(\tilde{g}) \equiv r \,({\rm mod}\,p^{k})$.
If $g$ is not congruent to the identity matrix modulo $p$, then the same argument as above shows that there are
\begin{equation} \label{numberofliftsnotI} p^{3(k-e)} \end{equation}
lifts of $g$. If $g$ is congruent to the
identity matrix modulo $p$, we have to count the number of matrices
\begin{eqnarray*}
\tilde{g} = \begin{pmatrix} 1 + k_1 p^e & k_2 p^e \\ k_3 p^e & 1 + k_4 p^e \end{pmatrix}
\end{eqnarray*}
such that $p^{2e} (k_1 k_4 - k_2 k_3) \equiv r \,({\rm mod}\,p^{k})$,
where $0 \leq k_i < p^{k-e}$.
If
$$
r \not\equiv 0 \,({\rm mod}\,\min{(p^{k}, p^{2e})}),
$$
there are no lifts,
and we suppose that  $ r \equiv 0 \,({\rm mod}\,\min{(p^{k}, p^{2e})})$.
Let $v = \min_i{v_p(k_i)}$, and write $k_i = p^v k_i'$ where $0 \leq v < k-e$ and $0 \leq k_i' < p^{k-e-v}$.
The congruence above rewrites as \begin{equation}
\label{newcongruence} p^{2e+v} (k_1' k_4' - k_2' k_3') \equiv r \,({\rm mod}\,p^{k}). \end{equation}
If $2e+v \geq k$, \eqref{newcongruence} has $p^{4(k-e-v)}$ solutions when $ r \equiv 0 \,({\rm mod}\, p^{k})$
and no solutions otherwise.  If $2e + v < k$, assume that $r \equiv 0 \,({\rm mod}\,(p^{2e+v}))$
(otherwise \eqref{newcongruence} has no solutions).
Writing $r = r' p^{2e + v}$, \eqref{newcongruence}
rewrites as $k_1' k_4' - k_2' k_3' \equiv r' \,({\rm mod}\,p^{k-2e-v})$ and this leads to
$p^e p^{3(k-e-v)}$ solutions $k_1', k_2', k_3', k_4'$.
Then, the number of lifts of the
identity matrix from $C_r(p^e)$ to $C_r(p^k)$ is bounded by
\begin{equation}\label{numberofliftsI}
\begin{aligned}
\sum_{\substack{v=0\\ 2e+v < k}}^{k-e-1} p^e p^{3(k-e-v)}
+ \sum_{\substack{v=0\\ 2e+v \geq k}}^{k-e-1}  p^{4(k-e-v)}
& \leq p^{3(k-e)} p^e \frac{p^3}{p^3 -1} +
p^{4e} \frac{p^4}{p^4-1}
\\
& \leq p^{3(k-e)} p^{4e+1} .
\end{aligned}
\end{equation}
Then, applying \eqref{numberofliftsI}, we have that
\begin{align*}
\frac{|C_r(m)|}{|G(m)|}
& \leq \frac{|C_r(m')|}{|G(m')|} \prod_{p \mid m}
\frac{p^{3(v_p(m)-v_p(m'))} p^{4v_p(m')+1}}{p^{4(v_p(m)-v_p(m'))}}
\\
& = \frac{|C_r(m')|}{|G(m')|} \frac{1}{\varphi(m)} \prod_{p \mid m} p^{v_p(m')-1} p^{4v_p(m')+1} (p-1)
\\
& \ll_E \frac{|C_r(m')|}{|G(m')|} \frac{1}{\varphi(m)}\cdot
\end{align*}
Finally we suppose that $|C_r(m)|\not=0$ and prove the lower bound in \eqref{bound2}.
Denoting by $C_r(m')_{\not\equiv}$ the subset of $C_r(m')$ consisting of matrices not equivalent
to the identity matrix modulo $p$ ($C_r(m')_{\not\equiv}$ is not empty since $|C_r(m)|\not=0$),
and applying \eqref{numberofliftsnotI}, we have that
\begin{align*}
\frac{|C_r(m)|}{|G(m)|}
& \geq \frac{|C_r(m')_{\not\equiv}|}{|G(m')|} \prod_{p \mid m}
\frac{p^{3(v_p(m)-v_p(m'))}}{p^{4(v_p(m)-v_p(m'))}}
\\
& = \prod_{p^k \| m} \frac{1}{p^{k-1} (p-1)} \prod_{p \mid m} \frac{(p-1) p^{v_p(m')}}{p} \frac{|C_r(m')_{\not\equiv}|}{|G(m')|}
\\
& \gg_E \frac{1}{\varphi(m)},
\end{align*}
and the lower bound in \eqref{bound2} follows from the last two inequalities.
\end{proof}

\begin{theorem}\label{thm-upperbound}
Let $E$ be an elliptic curve over $\Q$ without CM. Let $r\ge 0$ be an integer,
and let $n=dm$ be any positive integer with $(d, M_E)=1$ and $m\mid {M_E}^\infty$.
\par
{\rm (i)}
We have  that
$$
\left|\left\{ p \leq x \, : \, n_E(p) \equiv r  \,({\rm mod}\,n) \right\}\right|
\ll_E \frac{\Li(x)}{\varphi(n)} + x \exp\Big\{\!- A n^{-2} \sqrt{\log{x}}\Big\}
$$
uniformly for $\log{x} \gg n^{12} \log{n}$, where the implied constants depend only on
the elliptic curve $E$ and $A$ is a positive absolute constant.
\par
{\rm (ii)}
Assuming the GRH for the Dedekind zeta functions of the number fields $L_n/\Q$, we have  that
$$\left|\left\{ p \leq x \, : \, n_E(p) \equiv r  \,({\rm mod}\,n) \right\}\right|
\ll_E \frac{\Li(x)}{\varphi(n)} + n^3 x^{1/2} \log{(nx)}.$$
\par
{\rm (iii)}
Assuming the GRH for the Dedekind zeta functions of the number fields $L_n/\Q$,
we have that
\begin{eqnarray*}
\left|\left\{ p \leq x \, : \, n_E(p) \equiv r  \,({\rm mod}\,n) \right\}\right|
\ll_E \frac{\Li(x)}{\varphi(n)}
\end{eqnarray*}
holds uniformly for $n \leq x^{1/8}/\log x$,
where the implied constant depends only on the elliptic curve $E$.
\par
Further if $r=0$ or $(n, M_E)=1$,
then the condition $n \leq x^{1/8}/\log x$ in the third assertion can be relaxed to $n \leq x^{1/5}/\log x$
and the term $n^3 x^{1/2} \log(nx)$  in the second can be replaced by $n^{3/2} x^{1/2} \log(nx)$.
\end{theorem}

\begin{proof}
It follows from the estimates of Lemma \ref{lemma-upper-lower-bounds} that
\begin{eqnarray*}
\frac{|C_r(m)|}{|G(m)|} \bigg( \prod_{\ell^k\| d} \frac{|C_r(\ell^k)|}{|\GL_2(\Z / \ell^k \Z)|} \bigg)
\ll_E \frac{1}{\varphi(d)} \frac{1}{\varphi(m)} = \frac{1}{\varphi(n)},
\end{eqnarray*}
and first two statements are obtained by using this upper
bound in the estimates of Proposition \ref{CDT-bothcases} for
$$\pi_{C_r(n)}(x, L_{n}/\Q) =\left|\left\{ p \leq x \, : \, n_E(p) = p + 1 - a_E(p) \equiv r  \,({\rm mod}\,n) \right\}\right|.$$

\par

We now prove (iii).
If $|C_r(m)|=0$, Proposition \ref{CDT-bothcases} implies trivially the required inequality, and we
suppose that $|C_r(m)|\not=0$.
Clearly, it is sufficient to show that
\begin{eqnarray} \label{lowerbound-log}
\frac{1}{\varphi(n) \log_2n}
\ll_E \frac{|C_r(m)|}{|G(m)|} \bigg( \prod_{\ell^k\| d} \frac{|C_r(\ell^k)|}{|\GL_2(\Z / \ell^k \Z)|} \bigg)
\ll_E \frac{1}{\varphi(n)}\cdot
\end{eqnarray}
It follows from Lemma \ref{lemma-upper-lower-bounds}  that
\begin{eqnarray} \label{bound1}
\frac{1}{\varphi(d)} \prod_{\ell\mid d} \frac{\ell-2}{\ell-1}
\le \prod_{\ell^k\| d} \frac{|C_r(\ell^k)|}{|\GL_2(\Z / \ell^k \Z)|}
\ll \frac{1}{\varphi(d)},
\end{eqnarray}
and the lower bound of \eqref{lowerbound-log} follows from \eqref{bound1}, \eqref{bound2} and the estimate
\begin{align*}
\prod_{\ell\mid d} \frac{\ell-2}{\ell-1}
& \ge \prod_{\ell\mid n} \frac{\ell-2}{\ell-1}
\gg \frac{1}{\log_2n}\cdot
\end{align*}
This completes the proof of the Theorem.
\end{proof}

\vskip 10mm

\section{Rosser-Iwaniec's linear sieve formulas}\label{RosserIwaniec}

\smallskip

We state in this section the Rosser-Iwaniec linear sieve \cite[Theorem 1]{Iwaniec80},
which will be used in the proof of Theorem \ref{UB.Pseudoprime.thm}.
It is worth indicating that the Selberg linear sieve \cite[Theorem 8.4]{HR74} cannot be applied for our purpose
since the condition $(\Omega_2(1, L))$ of Selberg's linear sieve (see \cite[page 228]{HR74}) is not satisfied by the function $w_y(\ell)$. But the corresponding condition  $(\Omega_1)$ of the Rosser-Iwaniec's sieve
is satisfied by the $w_y(\ell)$
(see \eqref{X.wy.r} below).

Let ${\mathcal A}$ be a finite sequence of integers and ${\mathcal P}$ a set of prime numbers.
As usual, we write the sieve function
$$
S({\mathcal A}, {\mathcal P}, z)
:= |\{a\in {\mathcal A} : (a, P(z))=1\}|,
$$
where
\begin{equation}\label{defPz}
P(z)
:= \prod_{p<z, \, p\in {\mathcal P}} p.
\end{equation}
Let ${\mathcal B}={\mathcal B}({\mathcal P})$ denote the set of all positive squarefree
integers supported on the primes of ${\mathcal P}$. For each $d\in {\mathcal B}$, define
$$
{\mathcal A}_d:=\{a\in {\mathcal A} : a\equiv 0\,({\rm mod}\,d)\}.
$$
We assume that ${\mathcal A}$ is well distributed over arithmetic progressions $0\,({\rm mod}\,d)$
in the following sense:
There are a convenient approximation $X$ to $|{\mathcal A}|$
and a multiplicative function $w(d)$ on ${\mathcal B}$ verifying
\footnote{Since we need \eqref{defrAd} below only for $d\mid P(z)$, we are freely to define $w(p)=0$ for $p\notin {\mathcal P}$.}
$$
0<w(p)<p
\qquad(p\in {\mathcal P})
\leqno(A_0)
$$
such that
\par
(i) the ``remainders''
\begin{equation}\label{defrAd}
r({\mathcal A}, d)
:= |{\mathcal A}_d| - \frac{w(d)}{d}X
\qquad(d\in {\mathcal B})
\end{equation}
are small on average over the divisors $d$ of $P(z)$;
\par

(ii)
there exists a constant $K\ge 1$ such that
$$
\frac{V(z_1)}{V(z_2)}
\le \frac{\log z_2}{\log z_1}\bigg(1+\frac{K}{\log z_1}\bigg)
\qquad
(2\le z_1<z_2),
\leqno(\Omega_1)
$$
where
$$
V(z)
:= \prod_{p<z} \bigg(1-\frac{w(p)}{d}\bigg).
$$

\smallskip

The next result is the well known theorem of Iwaniec \cite[Theorem 1]{Iwaniec80}.

\begin{lemma}\label{RISieve}
Under the hypotheses $(A_0)$, \eqref{defrAd} and $(\Omega_1)$, we have
$$
S({\mathcal A}, {\mathcal P}, z)
\le XV(z)\{F(s)+E\}+2^{\varepsilon^{-\gamma}}R({\mathcal A}, M, N),
$$
where $0<\varepsilon<\tfrac{1}{8}$,
$s:=(\log MN)/\log z$, $E\ll \varepsilon s^2 {\rm e}^K+\varepsilon^{-8} {\rm e}^{K-s}(\log MN)^{-1/3}$ and
$$
F(s)=\frac{2{\rm e}^\gamma}{s}
\quad(0<s\le 3),
\qquad
V(z) := \prod_{p<z} \bigg(1-\frac{w(p)}{d}\bigg).
$$
The second error term $R({\mathcal A}, M, N)$ has the form
$$
R({\mathcal A}, M, N)
:= \sum_{\substack{m<M, n<N\\ mn\mid P(z)}} a_mb_n r({\mathcal A}, mn),
$$
where the coefficients $a_m, b_n$ are bounded by 1 in absolute value
and depend at most on $M, N, z$ and $\varepsilon$.
\end{lemma}

\vskip 10mm

\section{Proof of Theorem \ref{UB.Pseudoprime.thm}}

As in \cite{CLS09}, introduce
$$
L:=\prod_{y\le \ell<z} \ell
$$
and
\begin{align*}
{\mathcal S}(x, y, z)
& := \{p\le x \, : \, (n_E(p), L)=1\},
\\
{\mathcal T}(x, y, z)
& := \{p\le x \, : \, (n_E(p), L)>1, \; b^{n_E(p)}\equiv b\,({\rm mod}\, n_E(p))\}.
\end{align*}
Clearly
\begin{equation}\label{PseudoST}
Q_{E, b}(x)
\le |{\mathcal S}(x, y, z)|+|{\mathcal T}(x, y, z)|.
\end{equation}

\smallskip

First we estimate $|{\mathcal S}(x, y, z)|$.

\begin{lemma}\label{UpperBoundS}
Let $E$ be an elliptic curve over $\Q$ without CM and $b\ge 2$ be an integer.
For any $\varepsilon$, there is a constant $y_0=y_0(E, b, \varepsilon)$ such that
\par
{\rm (i)}
We have
\begin{equation}\label{UB.S}
|{\mathcal S}(x, y, z)|
\le ({\rm e}^\gamma  + \varepsilon) \frac{x\log y}{(\log x)\log z}
\end{equation}
uniformly for $y_0\le y\le z\le (\log x)^{1/24}/\log_2x$.

\vskip 1mm
\par
{\rm (ii)}
If we assume the GRH, we have
\begin{equation}\label{UB.S.GRH}
|{\mathcal S}(x, y, z)|
\le ({\rm e}^\gamma + \varepsilon) \frac{x\log y}{(\log x)\log z}
\end{equation}
uniformly for $y_0\le y\le z\le x^{1/10}/(\log x)^4$.
\end{lemma}

\begin{proof}
We shall sieve
$$
{\mathscr A}
:= \{n_E(p) \,:\, p\le x\}
$$
by
$$
{\mathscr P}_y
:= \{p \,: \, p\ge y\}.
$$
By definition,
$
|{\mathcal S}(x, y, z) |
= S({\mathscr A}, {\mathscr P}_y, z)$
for all $1\le y\le z\le x$.

Without loss of generality, we can suppose that $y_0\ge M_E+b$.
Thus we have $(d, M_E)=1$ for all $d\in {\mathcal B}({\mathscr P}_y)$.
Using Proposition \ref{CDT-bothcases} (with the improved error term discussed in the
remark following the proposition under the GRH) and \eqref{SizeCrelloverGell},
we get that
\begin{equation}\label{CardAd}
|{\mathscr A}_d|= \frac{w_y(d)}{d}X + r({\mathscr A}, d)
\end{equation}
for all $d\in {\mathcal B}({\mathscr P}_y)$, with
\begin{equation}\label{X.wy.r}
\begin{aligned}
X
& = \Li(x),
\\\noalign{\vskip 1mm}
w_y(\ell)
& = \frac{\ell(\ell^2-2)}{(\ell - 1)(\ell^2-1)}
\quad(\ell\in {\mathscr P}_y),
\\\noalign{\vskip 1mm}
|r({\mathscr A}, d)|
& \ll_E \begin{cases}
x \,{\rm e}^{-Ad^{-2}\sqrt{\log x}}  & \text{$(d\le (\log x)^{1/12}/\log_2x)$},
\\\noalign{\vskip 1mm}
d^{3/2}x^{1/2}\log(dx)                   & \text{(under the GRH)},
\end{cases}
\end{aligned}
\end{equation}
where $A>0$ is a positive absolute constant.

In order to apply Lemma \ref{RISieve}, we must show that $w_y(\ell)$ satisfies conditions $(A_0)$ and $(\Omega_1)$.
The former is obvious, and
we now check the latter.
Writing
\begin{equation}\label{VyV1}
V_y(z)
:= \prod_{p<z} \bigg(1-\frac{w_y(p)}{p}\bigg)^{-1},
\end{equation}
then
$$
\frac{V_y(z_1)}{V_y(z_2)}
\le \frac{V_1(z_1)}{V_1(z_2)}
$$
for all $z_2>z_1\ge 2$.
On the other hand,
by using the prime number theorem, it follows that
\begin{equation}\label{V1}
\begin{aligned}
V_1(z)
& = \prod_{p<z}
\bigg(1-\frac{w_1(p)}{p}\bigg)
\\
& = \prod_{p<z} \bigg(1-\frac{1}{p}\bigg)
\prod_{p<z} \bigg(1-\frac{p^2-p-1}{(p-1)^3(p+1)}\bigg)
\\
& = \bigg\{1+O\bigg(\frac{1}{\log z}\bigg)\bigg\} \frac{C {\rm e}^{-\gamma}}{\log z},
\end{aligned}
\end{equation}
where $\gamma$ is the Euler constant and
$$
C := \prod_{p} \bigg(1-\frac{p^2-p-1}{(p-1)^3(p+1)}\bigg).
$$
Clearly this implies that for any $2\le z_1<z_2$
\begin{equation}\label{Vz1/Vz2}
\frac{V_1(z_1)}{V_1(z_2)}
= \frac{\log z_2}{\log z_1} \bigg\{1+O\bigg(\frac{1}{\log z_1}\bigg)\bigg\},
\end{equation}
and \eqref{VyV1} and \eqref{Vz1/Vz2} show that the condition $(\Omega_1)$ is satisfied.
Therefore we can apply Lemma \ref{RISieve} to write
\begin{equation}\label{S2}
S({\mathscr A}, {\mathscr P}_y, z)
\le ({\rm e}^\gamma+\varepsilon) XV_y(z) + R_{\mathcal S},
\end{equation}
where
$$
R_{\mathcal S} := \sum_{\substack{d<z^2\\ d\mid P(z)}} 2^{\omega(d)} |r({\mathscr A}, d)|.
$$

In view of the bounds for  $|r({\mathscr A}, d)|$ of \eqref{X.wy.r}, we can deduce that
\begin{equation}\label{S3}
R_{\mathcal S}
\ll x/(\log x)^3
\end{equation}
for all
\begin{equation}\label{Cond.z.1}
z\le \begin{cases}
(\log x)^{1/24}/\log_2x & \text{(unconditionally)},
\\\noalign{\vskip 1,5mm}
x^{1/10}/(\log x)^4        & \text{(under GRH)}.
\end{cases}
\end{equation}

On the other hand,
in view of \eqref{V1}, we have for any $z>y$,
\begin{equation}\label{S4}
V_y(z)
= \frac{V_1(z)}{V_1(y)}
= \bigg\{1+O\bigg(\frac{1}{\log y}\bigg)\bigg\} \frac{\log y}{\log z}\cdot
\end{equation}
Inserting \eqref{S3} and \eqref{S4} into \eqref{S2}, we obtain the required results.
\end{proof}

\medskip

In order to estimate $|{\mathcal T}(x, y, z)|$, we need to prove a preliminary result.
For integers $b\ge 2$ and $d\ge 1$,
denote by ${\rm ord}_d(b)$ the multiplicative order of $b$ modulo $d$
(i.e. the smallest positive integer $k$ with $b^k\equiv 1\,({\rm mod}\,d)$).

\begin{lemma}\label{MulOrder}
For all $t\ge 1$, we have
\begin{align}
\sum_{\ell\ge t} \frac{1}{\ell {\rm ord}_\ell(b)}
& \ll_b \frac{1}{t^{1/2}},
\label{MulOrder1}
\\
\sum_{\ell {\rm ord}_\ell(b)\ge t} \frac{1}{\ell {\rm ord}_\ell(b)}
& \ll_b \frac{1}{t^{1/3}}\cdot
\label{MulOrder2}
\end{align}
\end{lemma}

\begin{proof}
Let $0<\eta<1$ be a parameter to be choosen later.
We have
\begin{equation}\label{Nord}
\sum_{\substack{\ell\\ {\rm ord}_\ell(b)=m}} 1
\le \sum_{\ell\mid (b^m-1)} 1
\le \frac{\log(b^m-1)}{\log 2}
\le \frac{\log b}{\log 2} m.
\end{equation}
Thus
$$
\sum_{\substack{\ell\le u\\ \text{ord}_{\ell}(b)<\ell^{\eta}}} \frac{1}{\text{ord}_{\ell}(b)}
= \sum_{m\le u^{\eta}} \frac{1}{m}
\sum_{\substack{\ell\le u\\ \text{ord}_{\ell}(b)=m}} 1
\le \sum_{m\le u^{\eta}} \frac{\log b}{\log 2}
\ll_{b, \eta} u^{\eta}.
$$
A simple partial summation leads to
$$
\sum_{\substack{\ell\ge t\\ \text{ord}_{\ell}(b)<\ell^{\eta}}} \frac{1}{\ell \text{ord}_{\ell}(b)}
= \int_t^\infty \frac{1}{u} \d \bigg(\sum_{\substack{\ell\le u\\ \text{ord}_{\ell}(b)<\ell^{\eta}}} \frac{1}{\text{ord}_{\ell}(b)}\bigg)
\ll_{b, \eta} \frac{1}{t^{1-\eta}}\cdot
$$
On the other hand, we have trivially
$$
\sum_{\substack{\ell\ge t\\ \text{ord}_{\ell}(b)\ge \ell^{\eta}}} \frac{1}{\ell \text{ord}_{\ell}(b)}
\ll \sum_{\ell\ge t} \frac{1}{\ell^{1+\eta}}
\ll_\eta \frac{1}{t^{\eta}}\cdot
$$
Combining these estimates and taking $\eta=\tfrac{1}{2}$,
we obtain \eqref{MulOrder1}.

Similarly we have
\begin{align*}
\sum_{\substack{\ell \text{ord}_{\ell}(b)\ge t\\ \text{ord}_{\ell}(b)<\ell^{\eta}}} \frac{1}{\ell \text{ord}_{\ell}(b)}
& \le \sum_{\substack{\ell\ge t^{1/(1+\eta)}\\ \text{ord}_{\ell}(b)<\ell^{\eta}}} \frac{1}{\ell \text{ord}_{\ell}(b)}
\ll_{b, \eta} \frac{1}{t^{(1-\eta)/(1+\eta)}},
\\
\sum_{\substack{\ell \text{ord}_{\ell}(b)\ge t\\ \text{ord}_{\ell}(b)\ge \ell^{\eta}}} \frac{1}{\ell \text{ord}_{\ell}(b)}
& = \sum_{k\ge 1}
\sum_{\substack{\ell \text{ord}_{\ell}(b)\ge t\\ 2^{k-1}\ell^{\eta}\le \text{ord}_{\ell}(b)<2^k\ell^{\eta}}}
\frac{1}{\ell \text{ord}_{\ell}(b)}
\\
& \ll \sum_{k\ge 1} \frac{1}{2^k}
\sum_{\ell\ge (2^{-k}t)^{1/(1+\eta)}}
\frac{1}{\ell^{1+\eta}}
\\
& \ll_\eta \frac{1}{t^{\eta/(1+\eta)}}\cdot
\end{align*}
The inequality \eqref{MulOrder2} follows from these estimates with the choice of $\eta=\tfrac{1}{2}$.
\end{proof}

We now estimate $|{\mathcal T}(x, y, z)|$.

\begin{lemma}\label{UpperBoundT}
Let $E$ be an elliptic curve over $\Q$ without CM and $b\ge 2$ be an integer.
Then there is a constant $y_0=y_0(E, b)$ and a positive absolute constant $A$ such that
\par
{\rm (i)}
We have
\begin{equation}\label{UB.T}
|{\mathcal T}(x, y, z)|\ll_{E, b} \Li(x) \frac{\log_2 z}{y^{1/2}} + x \exp\Big\{\!-Az^{-4}\sqrt{\log x}\Big\}
\end{equation}
uniformly for
\begin{equation}\label{Condition.xyz}
y_0\le y<z\le (\log x)^{1/24}/\log_2x.
\end{equation}
\par
{\rm (ii)}
If we assume the GRH, we have
\begin{equation}\label{UB.T.GRH}
|{\mathcal T}(x, y, z)|\ll_{E, b} \Li(x) \frac{\log_2 z}{y^{1/2}} + z^7 x^{1/2}
\end{equation}
uniformly for
\begin{equation}\label{Condition.xyz.GRH}
y_0\le y<z.
\end{equation}
\par
The implied constants depend on $E$ and $b$ only.
\end{lemma}

\begin{proof}
If $n_E(p)$ is a pseudoprime to base $b$ and $d\mid n_E(p)$ with $(d, b)=1$, then
$$
d\mid n_E(p)\mid b(b^{n_E(p)-1}-1)
\;\Rightarrow\;
d\mid (b^{n_E(p)-1}-1)
\;\Rightarrow\;
b^{n_E(p)-1}\equiv 1\,({\rm mod}\,d).
$$
Using Fermat's little theorem, it follows that
\begin{equation}\label{MO.Fermat}
n_E(p)\equiv 0\,({\rm mod}\,d),
\quad
n_E(p)\equiv 1\,({\rm mod}\,\text{ord}_d(b)),
\quad
(d, \text{ord}_d(b))= 1.
\end{equation}
By the Chinese remainder theorem, there is an integer $r_{b, d}\in \{1, \dots, d \text{ord}_d(b)\}$ such that
$n_E(p)\equiv r_{b, d}\,({\rm mod}\,d \text{ord}_d(b))$.

Clearly for each $p\in {\mathcal T}(x, y, z)$,
there is a prime $\ell$ such that
\begin{equation}\label{pseudoprime1}
y\le \ell<z,
\qquad
\ell\mid (L, n_E(p))
\qquad\text{and}\qquad
n_E(p)\mid b^{n_E(p)}-b.
\end{equation}
Applying \eqref{MO.Fermat} with $d=\ell$, we have
\begin{align*}
|{\mathcal T}(x, y, z)|
& \le \sum_{y<\ell\le z}
\sum_{\substack{p\le x\\ n_E(p)\equiv r_{b, \ell} ({\rm mod}\,\ell\text{ord}_{\ell}(b))}} 1
\\
& = \sum_{y<\ell\le z} \pi_{C_{r_{b, \ell}}} ( x, L_{\ell\text{ord}_{\ell}(b)}/\Q ).
\end{align*}
Then, using (i) and (ii) of Theorem \ref{thm-upperbound} with the bound $\varphi(n)\gg n/\log_2n$,
we have that
\begin{equation}\label{T1}
|{\mathcal T}(x, y, z)|
\ll_E \Li(x) (\log_2 z) \sum_{y<\ell\le z} \frac{1}{\ell \text{ord}_{\ell}(b)}
+ R_{\mathcal T},
\end{equation}
where
\begin{equation}\label{T7}
\begin{aligned}
R_{\mathcal T}
& := \begin{cases}\displaystyle
\sum_{y<\ell\le z} x \exp\Big\{\!-A\ell^{-4}\sqrt{\log x}\Big\}
& \text{($z \leq (\log{x})^{1/24}/ \log_2{x}$)}
\\\noalign{\vskip 2mm}
\displaystyle
\sum_{y<\ell\le z} \ell^6 x^{1/2}\log(\ell^2 x)
& \text{(under the GRH)}
\end{cases}
\\\noalign{\vskip 1mm}
& \ll \begin{cases}\displaystyle
x \exp\Big\{-Az^{-4}\sqrt{\log x}\Big\}
& \text{($z \leq (\log{x})^{1/24}/ \log_2{x}$),}
\\\noalign{\vskip 1mm}
\displaystyle
z^7 x^{1/2}
& \text{(under the GRH).}
\end{cases}
\end{aligned}
\end{equation}

The required results follow from \eqref{T1}, \eqref{T7} and \eqref{MulOrder1} of Lemma \ref{MulOrder}.
\end{proof}

Taking, in Lemmas \ref{UpperBoundS} and \ref{UpperBoundT}
\begin{align*}
y
& =\begin{cases}
(\log_2 x)^2\log_3 x    & \text{(unconditionally)},
\\\noalign{\vskip 1mm}
(\log x)^2\log_2 x         & \text{(under the GRH)},
\end{cases}
\\
z
& =\begin{cases}
(\log x)^{1/24}/\log_2 x    & \text{(unconditionally)},
\\\noalign{\vskip 1mm}
x^{1/14}/\log x                   & \text{(under the GRH)},
\end{cases}
\end{align*}
which satisfy \eqref{Cond.z.1} and \eqref{Condition.xyz},
and using the bounds of those lemmas in \eqref{PseudoST}, this proves  Theorem \ref{UB.Pseudoprime.thm}.

\vskip 10mm

\section{Proof of Theorem \ref{UB.Pseudoprime.NonCM}}   \label{Section-adapt-Pomerance}

We shall adapt Pomerance's method \cite{Pomerance81} to prove Theorem \ref{UB.Pseudoprime.NonCM}.

We divise the primes $p\le x$ such that $n_E(p)$ is pseudoprimes to base $b$ into four possibly overlapping classes:
\begin{itemize}
\item{$n_E(p)\le x/L(x)$;}

\vskip 0,5mm

\item{
there is $\ell\mid n_E(p)$ with $\text{ord}_\ell(b)\le L(x)$ and $\ell>L(x)^3$;}

\vskip 0,5mm

\item{
there is $\ell\mid n_E(p)$ with $\text{ord}_\ell(b)>L(x)$;}

\vskip 0,5mm

\item{$n_E(p)>x/L(x)$, for all $\ell\mid n_E(p)$, we have $\ell\le L(x)^3$;}
\end{itemize}
and denote by $S_1, \dots, S_4$ the corresponding contribution to $\pi_{E, b}^{\rm pseu}(x)$, respectively.

\vskip 1mm
\goodbreak

{A. \sl Estimate for $S_1$}

\vskip 0,5mm

In view of \eqref{Hasse1}, it follows that
\begin{equation}\label{UBS1}
S_1
\le \sum_{p\le 16x/L(x)} 1
\ll \frac{x}{L(x)}\cdot
\end{equation}

\vskip 1mm

{B. \sl Estimate for $S_2$}

\vskip 0,5mm

Clearly
$$S_2
\le \sum_{\substack{\ell>L(x)^3\\ \text{ord}_\ell(b)\le L(x)}}
\sum_{\substack{p\le x\\ \ell\mid n_E(p)}} 1.
$$
Using (iii) of Theorem \ref{thm-upperbound} with $r=0$ and \eqref{Nord}, we
 deduce that the contribution of
$L(x)^3<\ell\le x^{1/5}/\log x$ to $S_2$ is
$$
\ll_E \sum_{\substack{L(x)^3<\ell\le x^{1/5}/\log x\\ \text{ord}_\ell(b)\le L(x)}} \frac{\Li(x)}{\varphi(\ell)}
\ll_E \frac{x}{L(x)^3} \sum_{\text{ord}_\ell(b)\le L(x)} 1
\ll_{E, b} \frac{x}{L(x)}\cdot
$$
Furthermore,
using Hypothesis \ref{H} with $\delta<\tfrac{1}{5}$, we have
\begin{align*}
\sum_{\substack{x^{1/5}/\log x<\ell \\ \text{ord}_\ell(b)\le L(x)}}
\sum_{\substack{p\le x\\ \ell\mid n_E(p)}} 1
& \le \sum_{\substack{x^{1/5}/\log x<\ell\le 2x\\ \text{ord}_\ell(b)\le L(x)}}
\sum_{m\le 2x/\ell}
\sum_{\substack{p\le x\\ n_E(p)=m\ell}} 1
\\
& \ll_E \sum_{\substack{x^{1/5}/\log x<\ell\le 2x\\ \text{ord}_\ell(b)\le L(x)}}
\sum_{m\le 2x/\ell} (m\ell)^{\delta}
\\\noalign{\vskip -1,5mm}
& \ll_E \sum_{\substack{x^{1/5}/\log x<\ell\le 2x\\ \text{ord}_\ell(b)\le L(x)}}
\frac{x^{1+\delta}}{\ell}
\\
& \ll_{E, b} x^{4/5+\delta} L(x)^3,
\end{align*}
using \eqref{Nord}.

Combining these estimates yields
\begin{equation}\label{UBS2b}
S_2
\ll_{E, b} \frac{x}{L(x)}\cdot
\end{equation}

\vskip 1mm

{C. \sl Estimate for $S_3$}

\vskip 0,5mm

If $p$ is counted in $S_3$,
then there is $\ell\mid n_E(p)$ with ${\rm ord}_\ell(b)>L(x)$
(which implies $\ell>L(x)>b$).
Applying \eqref{MO.Fermat} with $d=\ell$,
there is an integer $r_{b, \ell}\in \{1, \dots, \ell \text{ord}_\ell(b)\}$ such that
$n_E(p)\equiv r_{b, \ell}\,({\rm mod}\,d \text{ord}_d(b))$.
Since $n_E(p)\le p+1+2\sqrt{p}\le 4p\le 4x$, we must have $\ell \text{ord}_\ell(b)\le 4x$.
Thus
\begin{equation}\label{UBS3a}
S_3
\le \sum_{\substack{\ell \text{ord}_\ell(b)\le 4x\\ \text{ord}_\ell(b)>L(x)}}
\sum_{\substack{p\le x\\ n_E(p)\equiv r_{b, \ell} ({\rm mod}\,\ell \text{ord}_\ell(b))}} 1.
\end{equation}
If $\ell\text{ord}_\ell(b)\le x^{1/8}/\log x$, then by Theorem \ref{thm-upperbound}(iii) $$
\sum_{\substack{p\le x\\ n_E(p)\equiv r_{b, \ell} ({\rm mod}\,\ell \text{ord}_\ell(b))}} 1 \ll_E \frac{\Li(x)}{\varphi(\ell \text{ord}_\ell(b))},
$$
and using again the bound $\varphi(n) \gg n / \log_2{n}$, the contribution of those $\ell$ to $S_3$ is bounded by
\begin{eqnarray*}
\sum_{\substack{\ell\text{ord}_\ell(b)\le x^{1/8}/\log x\\ \text{ord}_\ell(b)>L(x)}}
\frac{\Li(x)}{\varphi(\ell \text{ord}_\ell(b))}
& \ll_E & \frac{\Li(x)\log_2x}{L(x)} \sum_{\ell\text{ord}_\ell(b)\le x^{1/8}/\log x} \frac{1}{\ell}
\\
& \ll_E & \frac{\Li(x)(\log_2x)^2}{L(x)}\cdot
\end{eqnarray*}
With the help of Hypothesis \ref{H} with $\delta<\tfrac{1}{24}$
and \eqref{MulOrder2} of Lemma \ref{MulOrder},
the contribution of $x^{1/8}/\log x<\ell\text{ord}_\ell(b)\le 4x$ to $S_3$ is bounded by
\begin{align*}
&  \sum_{x^{1/8}/\log x<\ell\text{ord}_\ell(b)\le 4x}
\sum_{0\le m\le 4x/\ell \text{ord}_\ell(b)}
\sum_{\substack{p\le x\\ n_E(p)=r_{b, \ell}+m\ell \text{ord}_\ell(b)}} 1
\\
& \ll_E \sum_{x^{1/8}/\log x<\ell\text{ord}_\ell(b)\le 4x}
\sum_{0\le m\le 4x/\ell\text{ord}_\ell(b)}
(r_{b, \ell}+m\ell \text{ord}_\ell(b))^{\delta}
\\
& \ll_E \sum_{x^{1/8}/\log x<\ell\text{ord}_\ell(b)\le 4x}
\frac{x^{1+\delta}}{\ell \text{ord}_\ell(b)}
\\\noalign{\vskip 1mm}
& \ll_E x^{1+\delta-1/24}\log x.
\end{align*}

Inserting these estimates into \eqref{UBS3a}, we find that
\begin{equation}\label{UBS3b}
S_3
\ll_E \frac{x}{L(x)}\cdot
\end{equation}

\vskip 1mm

{D. \sl Estimate for $S_4$}

\vskip 0,5mm

In order to adapt the proof of \cite{Pomerance81} to the more general definition \eqref{qnq} of pseudoprimes (which
includes the case where $b$ and $n$ are not coprime), we
write $n_E(p)=n'_E(p) n''_E(p)$ with $n'_E(p) \mid b^\infty$ and $(n''_E(p), b)=1$.
Denote by $S_4'$ and $S_4''$ the contribution of $n'_E(p)>x^{2/3}$ and $n'_E(p) \le x^{2/3}$ to $S_4$, respectively.

By the Hasse bound (formulated as the statement of Hypothesis \ref{H} with $\delta=\tfrac{1}{2}$),
we have
\begin{align*}
S_4'
& \le \sum_{\substack{x^{2/3}<d\le 4x\\ d\mid b^\infty}} \sum_{\substack{m\le 4x/d\\ (m, b)=1}}
\sum_{\substack{p\le x\\ n'_E(p)=d, \, n''_E(p)=m}} 1
\\
& \ll_E \sum_{\substack{x^{2/3}<d\le 4x\\ d\mid b^\infty}} \sum_{m\le 4x/d}
(dm)^{1/2}
\\\noalign{\vskip -1mm}
& \leq \sum_{\substack{x^{2/3}<d\le 4x\\ d\mid b^\infty}}
\frac{x^{3/2}}{d}
\\
& \leq x^{5/6} (\log x)^{b}.
\end{align*}

If $p$ is counted in $S_4''$, then $n''_E(p)>x^{1/3}/L(x)$ and all prime factors of $n''_E(p)$ are $\le L(x)^3$.
Thus $n''_E(p)$ must have a divisor $d$ with
$x^{1/18}<d\le x^{1/17}$ and $(d, b)=1$.
Thus, by the comment following \eqref{MO.Fermat},   $n_E(p)\equiv r_{b, d} \, ({\rm mod}\,d \text{ord}_d(b))$
for some residue $r_{b,d}$, and by
Theorem \ref{thm-upperbound}, we have
\begin{align*}
S_4''
& \le \sum_{\substack{x^{1/18}<d\le x^{1/17}\\ (d, b)=1}}
\sum_{\substack{p\le x\\ n_E(p)\equiv r_{b, d} ({\rm mod}\,d \text{ord}_d(b))}} 1
\\
& \ll_E \sum_{x^{1/18}<d\le x^{1/17}} \frac{x}{d \text{ord}_d(b)}
\\\noalign{\vskip 1mm}
& \leq x \sum_{m\le x^{1/17}} \frac{1}{m}
\sum_{\substack{x^{1/18}<d\le x^{1/17}\\ \text{ord}_d(b)=m}}
\frac{1}{d}\cdot
\end{align*}
With the help of the following inequality (see \cite[Theorem 1]{Pomerance81})
$$
\sum_{\substack{d\le t\\ {\rm ord}_d(b)=m}} 1
\le \frac{t}{\sqrt{L(t)}}
\quad
(t\ge t_0(b), \, m\ge 1),
$$
a simple partial integration allows us to deduce that
\begin{align*}
\sum_{\substack{x^{1/18}<d\le x^{1/17}\\ \text{ord}_d(b)=m}} \frac{1}{d}
& = \int_{x^{1/18}}^{x^{1/17}} \frac{1}{t} \d \Big(\sum_{\substack{d\le t\\ \text{ord}_d(b)=m}} 1\Big)
\ll \frac{1}{L(x)^{1/37}},
\end{align*}
and $S''_4 \ll_E x (\log{x}) L(x)^{-1/37}$.
Thus
\begin{equation}\label{UBS4}
S_4     = S'_4 + S''_4
\ll_{E, b} \frac{x}{L(x)} + \frac{x\log x}{L(x)^{1/37}}
\le \frac{x}{L(x)^{1/38}}\cdot
\end{equation}

The statement of Theorem \ref{UB.Pseudoprime.NonCM} then follows from \eqref{UBS1}, \eqref{UBS2b}, \eqref{UBS3b} and \eqref{UBS4}.

\vskip 10mm

\section{Proof of Theorem \ref{UB.Pseudoprime.CM}}

First write
\begin{align*}
\pi_{E, b}^{\rm pseu}(x)
& = \sum_{\substack{p\le x\\ \text{$n_E(p)$ is pseudoprime to base $b$}}} 1
\\
& \le \sum_{\substack{n\le 4x\\ \text{$n$ is pseudoprime to base $b$}}} M_E(n).
\end{align*}
By using the Cauchy-Schwarz inequality, it follows that
\begin{equation}\label{CM1}
\begin{aligned}
\pi_{E, b}^{\rm pseu}(x)
& \le \Big(\pi_b^{\rm pseu}(4x) \Big)^{1/2} \Big( \sum_{n\le 4x} M_E(n)^2\Big)^{1/2}.
\end{aligned}
\end{equation}
To bound the second sum on the right-hand side of \eqref{CM1}, we use a result of Kowalski \cite{Kowalski06}
who proved that for a curve $E$ with complex multiplication  and for any $\varepsilon > 0$,
\begin{equation}\label{CM2}
\sum_{n\le 4x} M_E(n)^2
\ll \frac{x}{(\log x)^{1-\varepsilon}}.
\end{equation}
We remark that in \cite{Kowalski06}, there are no curves with complex multiplication defined over
$\Q$ as the field of complex multiplication must be included in the field of definition of
the elliptic curve. Then, \eqref{CM2} is first proven for the sequence
$\left\{ n_E(\frakp) = \# E(\F_\frakp) \right\}$
associated to $E$, where $\frakp$ runs over the primes of the CM field \cite[Theorem 5.4]{Kowalski06}.
This first result can then be used to deduce the upper bound \eqref{CM2} by separating the rational
primes into ordinary and supersingular primes of $E$, and by using \cite[Theorem 5.4]{Kowalski06} to obtain
\eqref{CM2} (see \cite[Proposition 7.4]{Kowalski06}).

Theorem \ref{UB.Pseudoprime.CM} then follows by replacing  \eqref{CM2} and \eqref{Pomerance1} in \eqref{CM1}.

\end{document}